\documentclass[12pt]{article}
\usepackage{e-jc}

\usepackage{amsthm,amsmath,amssymb,mathtools}

\usepackage{graphicx}

\usepackage[colorlinks=true,citecolor=black,linkcolor=black,urlcolor=blue]{hyperref}

\newcommand\R{\mathbb{R}}
\newcommand\Z{\mathbb{Z}}

\newcommand\C{\mathbb{C}}
\newcommand\eps{\varepsilon}

\theoremstyle{plain}
\newtheorem{theorem}{Theorem}
\newtheorem{lemma}[theorem]{Lemma}

\newtheorem{proposition}[theorem]{Proposition}

\theoremstyle{definition}

\theoremstyle{remark}
\newtheorem{remark}[theorem]{Remark}

\title{\bf A bound on partitioning clusters}

\author{Daniel Kane\\
\small Department of Mathematics, UC San Diego\\[-0.8ex]
\small 9500 Gilman Drive, La Jolla CA 92093-0112\\[-0.8ex] 
\small\tt dakane@ucsd.edu\\
\and
Terence Tao \\
\small Department of Mathematics, UCLA\\[-0.8ex]
\small 405 Hilgard Ave, Los Angeles CA 90095\\[-0.8ex]
\small\tt tao@math.ucla.edu
}

\date{\dateline{Feb 2, 2017}{Feb 2, 2017}}

\begin{document}

\maketitle


\begin{abstract}
Let $X$ be a finite collection of sets (or ``clusters'').  We consider the problem of counting the number of ways a cluster $A \in X$ can be partitioned into two disjoint clusters $A_1, A_2 \in X$, thus $A = A_1 \uplus A_2$ is the disjoint union of $A_1$ and $A_2$; this problem arises in the run time analysis of the ASTRAL algorithm in phylogenetic reconstruction.  We obtain the bound
$$ | \{ (A_1,A_2,A) \in X \times X \times X: A = A_1 \uplus A_2 \} | \leq |X|^{3/p} $$
where $|X|$ denotes the cardinality of $X$, and $p \coloneqq \log_3 \frac{27}{4} = 1.73814\dots$, so that $\frac{3}{p} = 1.72598\dots$.  Furthermore, the exponent $p$ cannot be replaced by any larger quantity.  This improves upon the trivial bound of $|X|^2$.  The argument relies on establishing a one-dimensional convolution inequality that can be established by elementary calculus combined with some numerical verification.

In a similar vein, we show that for any subset $A$ of a discrete cube $\{0,1\}^n$, the additive energy of $A$ (the number of quadruples $(a_1,a_2,a_3,a_4)$ in $A^4$ with $a_1+a_2=a_3+a_4$) is at most $|A|^{\log_2 6}$, and that this exponent is best possible.
\end{abstract}

\section{Introduction}

The purpose of this note is to establish the following theorem:

\begin{theorem}\label{main} Let $X$ be a finite collection of sets.  Then we have
\begin{equation}\label{aaa}
 | \{ (A_1,A_2,A) \in X \times X \times X: A = A_1 \uplus A_2 \} | \leq |X|^{3/p}
\end{equation}
where $|X|$ denotes the cardinality of $X$, $A = A_1 \uplus A_2$ denotes the assertion that $A$ is the disjoint union of $A_1$ and $A_2$, and
$p \coloneqq \log_3 \frac{27}{4} = 1.73814\dots$.  Furthermore, the exponent $p$ cannot be replaced by any larger quantity.
\end{theorem}

Note that $\frac{3}{p} = 1.72598\dots$.  Thus the inequality \eqref{aaa} improves upon the trivial bound
\begin{equation}\label{aaa-triv}
 | \{ (A_1,A_2,A) \in X \times X \times X: A = A_1 \uplus A_2 \} | \leq |X|^2
\end{equation}
that arises simply because there are $|X|^2$ pairs $(A_1,A_2)$, and this pair uniquely determines $A$.

Theorem \ref{main} has applications in analyzing the running time of several dynamic programming algorithms used in phylogenetic reconstruction literature.  Several published algorithms \cite{mrbzsw, mw, hl, bs} seek to find a median tree that minimizes the total distance to an input set of trees, for various definitions of a distance between two trees.   For example, a method called ASTRAL \cite{mrbzsw, mw} defines the distance between two trees as the number of quartet trees induced by one tree but not the other, and seeks to find the unrooted tree that minimizes the sum of this quartet distance to its input set of unrooted trees, a problem that turns out to be NP-Hard \cite{ls}.
To make this optimization problem tractable, ASTRAL uses a dynamic programming approach where the final tree is built by successively dividing each subset of leaves (called a cluster) $A$ into smaller clusters, $A'$ and $A \backslash A'$ while minimizing
the number of quartets in the input tree set that will have to be missing from any tree that includes $A'|A \backslash A'$ as a bipartition (i.e., a branch; note that a branch in an unrooted tree is just a bipartition of the leaves).
If all possible subsets $A'\subset A$ are considered when dividing $A$ to two subsets, the algorithm provably returns the optimal tree
(also, the optimal solution has been shown to enjoy statistical consistency under certain models of gene and species evolution \cite{mrbzsw}).
However, such a dynamic programming algorithm will have to explore the power set of the set of leaves and will thus require time exponential in the number of leaves.

To give a practical alternative, ASTRAL solves a constrained version of the problem where a set $X$ of clusters is defined in advance to constrain the search space; when dividing $A$, we only look for $A'\subset A$ such that $A'\in X$ and $A \backslash A' \in X$.
The set $X$ is defined heuristically, and the running time of ASTRAL should be defined asymptotically as a function of $|X|$.
Throughout the dynamic programming execution, ASTRAL considers all possible pairs of clusters $x,y\in X$ exactly once iff $x\cap y = \emptyset$ and $x \cup y \in X$.
Therefore, establishing the asymptotic running time of ASTRAL with regards to $|X|$ requires bounding the left-hand side of \eqref{aaa}
ASTRAL simply used the trivial $O(|X|^2)$ upper bound in their running time analysis. We can now improve that analysis to $O(|X|^{1.72598\dots})$.

We first demonstrate why the exponent $p$ is best possible.  Let $n$ be a large multiple of $3$, and let $X$ denote the collection of all sets $A \subset \{1,\dots,n\}$ whose cardinality $|A|$ is equal to either $n/3$ or $2n/3$.  Clearly
$$ |X| = \binom{n}{n/3} + \binom{n}{2n/3} = 2 \frac{n!}{(n/3)! (2n/3)!}.$$
On the other hand, if $A \in X$ has cardinality $2n/3$, then it can be partitioned in $\binom{2n/3}{n/3}$ ways into $A_1 \uplus A_2$ with $A_1,A_2 \in X$, and no partition is available when $A$ has cardinality $n/3$.  Thus the left-hand side of \eqref{aaa} is equal to
$$ \binom{n}{2n/3} \times \binom{2n/3}{n/3} = \frac{n!}{(n/3)! (n/3)! (n/3)!}.$$
Using the Stirling approximation $n! = n^n e^{-n+o(n)}$ as $n \to \infty$, we conclude that the left-hand side of \eqref{aaa} is equal to $\exp( n (\log 3 + o(1)))$, while $|S|$ is equal to $\exp( \frac{n}{3} (\log \frac{27}{4}+o(1)) )$, and on sending $n \to \infty$ we conclude that \eqref{aaa} can fail whenever $p > \log_3 \frac{27}{4}$.

Now we show why \eqref{aaa} holds for $p = \log_3 \frac{27}{4}$.  This will be a consequence of the following convolution estimate on the discrete unit cube $\{0,1\}^n$ (which we view as a subset of $\Z^n$ for the purposes of defining convolution):

\begin{theorem}[Convolution]\label{conv}  Let $n \geq 1$ be a natural number, and let $f,g,h: \{0,1\}^n \to \R$ be functions.  Set $p := \log_3 \frac{27}{4}$.  Then we have
\begin{equation}\label{po}
 f*g*h(1^n) \leq \|f\|_{\ell^p(\{0,1\}^n)} \|g\|_{\ell^p(\{0,1\}^n)} \|h\|_{\ell^p(\{0,1\}^n)}
\end{equation}
where $1^n$ denotes the element $(1,\dots,1)$ of $\{0,1\}^n$, $f*g*h$ denotes the convolution
$$ f*g*h(w) \coloneqq \sum_{x,y,z \in \{0,1\}^n: x+y+z = w} f(x) g(y) h(z)$$
and $\| \cdot \|_{\ell^p(\{0,1\}^n)}$ denotes the $\ell^p$ norm
$$ \|f\|_{\ell^p(\{0,1\}^n)} \coloneqq (\sum_{x \in \{0,1\}^n} |f(x)|^p)^{1/p}.$$
\end{theorem}

Let us now see why Theorem \ref{conv} implies Theorem \ref{main}.  We first observe that to prove Theorem \ref{main}, it suffices to do so under the additional assumption that all the sets in $X$ are finite. Indeed, if we let $\Omega \coloneqq \bigcup_{A \in X} A$ denote the union of the sets in $X$, then $X$ partitions $\Omega$ into at most $2^{|X|}$ cells.  Some of these cells may be infinite, but we may replace any such cell with a single point without affecting either the left or right-hand side of \eqref{aaa}.  After applying this replacement, every set in $X$ is now finite.

Without loss of generality, we may now assume that all the sets $A$ in $X$ are subsets of $\{1,\dots,n\}$ for some natural number $n \geq 1$.  We now define the functions $f,g,h: \{0,1\}^n \to \{0,1\}$ as follows.  For any $(a_1,\dots,a_n) \in \{0,1\}^n$, we set $f(a_1,\dots,a_n) = g(a_1,\dots,a_n) = 1$ when the set $\{ 1 \leq i \leq n: a_i = 1 \}$ lies in $X$, and $f(a_1,\dots,a_n) = g(a_1,\dots,a_n) = 0$ otherwise.  Similarly, we set $h(a_1,\dots,a_n)=1$ when the set $\{ 1 \leq i \leq n: a_i = 0 \}$ lies in $X$, and $h(a_1,\dots,a_n) = 0$ otherwise.  It is then easy to see that
$$ \|f\|_{\ell^p(\{0,1\}^n)} = \|g\|_{\ell^p(\{0,1\}^n)} = \|h\|_{\ell^p(\{0,1\}^n)} = |X|^{1/p}$$
and
$$ f*g*h(1^n) = | \{ (A_1,A_2,A) \in X \times X \times X: A = A_1 \uplus A_2 \} |$$
giving the claim.

\begin{remark}  Young's convolution inequality establishes \eqref{po} with $p$ replaced by $3/2$ (or any exponent less than $3/2$).  This corresponds to the trivial bound \eqref{aaa-triv}.  The ability to improve the exponents in Young's convolution inequality is reminiscent of the Kunze-Stein phenomenon \cite{kunze} in semisimple Lie groups, as well as the hypercontractivity inequality on the Boolean cube (see e.g. \cite{hyper}).  Indeed, the proof of Theorem \ref{conv} will be similar to the proof of hypercontractivity in that we will soon reduce matters to verifying the one-dimensional case $n=1$.
\end{remark}

\begin{remark}  The above argument in fact establishes the more general inequality
$$
 | \{ (A_1,A_2,A) \in X \times X_1 \times X_2: A = A_1 \uplus A_2 \} | \leq |X|^{1/p} |X_1|^{1/p} |X_2|^{1/p}
$$
whenever $X,X_1,X_2$ are finite collections of sets.
\end{remark}

The form of Theorem \ref{conv} is very amenable to an induction on dimension:

\begin{proposition}  Let $n_1,n_2 \geq 1$ be natural numbers.  If Theorem \ref{conv} holds for $n=n_1$ and $n=n_2$, then it holds for $n=n_1+n_2$.
\end{proposition}

\begin{proof}  Let $f,g,h: \{0,1\}^{n_1+n_2} \to \R$ be functions.  For any $x \in \{0,1\}^{n_1}$, let $f_x: \{0,1\}^{n_2} \to \R$ denote the function
$$ f_x(x') \coloneqq f(x,x')$$
for $x' \in \{0,1\}^{n_2}$ (where we identify the pair $(x,x')$ with an element of $\{0,1\}^{n_1+n_2}$ in the usual fashion).  Then we can write
$$ f*g*h(1^{n_1+n_2}) = \sum_{x,y,z \in \{0,1\}^{n_1}: x+y+z = 1^{n_1}} f_x * g_y * h_z(1^{n_2}).$$
Applying Theorem \ref{conv} for $n=n_2$ and the functions $f_x,g_y,h_z$, we thus have
$$ f*g*h(1^{n_1+n_2}) \leq \sum_{x,y,z \in \{0,1\}^{n_1}: x+y+z = 1^{n_1}} \| f_x \|_{\ell^p(\{0,1\}^{n_2})}
\| g_y \|_{\ell^p(\{0,1\}^{n_2})} \| h_z \|_{\ell^p(\{0,1\}^{n_2})}.$$
Applying Theorem \ref{conv} for $n=n_1$ and the functions $x \mapsto \| f_x \|_{\ell^p(\{0,1\}^{n_2})}$,
$y \mapsto \| g_y \|_{\ell^p(\{0,1\}^{n_2})}$, $z \mapsto \| h_z \|_{\ell^p(\{0,1\}^{n_2})}$,
we obtain
$$ f*g*h(1^{n_1+n_2}) \leq \|f\|_{\ell^p(\{0,1\}^{n_1+n_2}} \|g\|_{\ell^p(\{0,1\}^{n_1+n_2}} \|h\|_{\ell^p(\{0,1\}^{n_1+n_2}}$$
which gives Theorem \ref{conv} for $n=n_1+n_2$.
\end{proof}

From this proposition and induction, we see that to prove Theorem \ref{conv}, it suffices to do so in the one-dimensional case $n=1$.  We may normalize
$$\|f\|_{\ell^p(\{0,1\})} = \|g\|_{\ell^p(\{0,1\})} = \|h\|_{\ell^p(\{0,1\})} = 1,$$
so that we may write
\begin{align*}
f(0) &= a^{1/p} \\
f(1) &= (1-a)^{1/p} \\
g(0) &= b^{1/p} \\
g(1) &= (1-b)^{1/p} \\
h(0) &= c^{1/p} \\
h(1) &= (1-c)^{1/p}
\end{align*}
for some $0 \leq a,b,c \leq 1$.  The inequality \eqref{po} then simplifies to the elementary inequality
\begin{equation}\label{abc}
 (ab(1-c))^{1/p} + (bc(1-a))^{1/p} + (ca(1-b))^{1/p} \leq 1.
\end{equation}
Observe that equality is attained here when $(a,b,c) = (0,1,1), (1,0,1), (0,1,1), (2/3,2/3,2/3)$; the final case $(a,b,c) = (2/3,2/3,2/3)$ also reveals that the inequality \eqref{abc} fails if $p$ is replaced by any quantity larger than $\log_3 \frac{27}{4}$.  This is of course consistent with the second part of Theorem \ref{main}.

The fact that equality is attained in \eqref{abc} in four different locations seems to rule out any quick proof of \eqref{abc} using convexity-based methods such as Jensen's inequality.  Instead, we argue as follows.  First observe that when $a=0$, the left-hand side of \eqref{abc} simplifies to $(bc)^{1/p}$, and it is then clear that the inequality \eqref{abc} holds whenever $a=0$ and is strict unless $(a,b,c)=(0,1,1)$.  Next, we analyze the left-hand side of \eqref{abc} for $(a,b,c)$ close to $(0,1,1)$.  Writing $(a,b,c) = (\alpha, 1-\beta, 1-\gamma)$ for some small $\alpha,\beta,\gamma \geq 0$, we can write the left-hand side of \eqref{abc} as
$$
(\alpha \gamma)^{1/p} + ( (1-\beta)(1-\gamma)(1-\alpha) )^{1/p} + (\alpha \beta)^{1/p}.$$
For $\alpha,\beta,\gamma$ small enough, we have
$$(1-\beta)(1-\gamma)(1-\alpha) \leq 1 - \frac{1}{2}(\alpha+\beta+\gamma)$$
(say), which by the concavity of $x \mapsto x^{1/p}$ implies that
$$( (1-\beta)(1-\gamma)(1-\alpha) )^{1/p} \leq 1 - \frac{1}{2p}(\alpha+\beta+\gamma).$$
On the other hand, from the arithmetic mean-geometric mean inequality we certainly have
$$ (\alpha \gamma)^{1/p}, (\alpha \beta)^{1/p} \leq (\alpha+\beta+\gamma)^{2/p}.$$
Since $p<2$, we conclude that the inequality \eqref{abc} holds whenever $\alpha+\beta+\gamma$ is sufficiently small, or equivalently when $(a,b,c)$ is sufficiently close to $(0,1,1)$.  Since both sides of \eqref{abc} depend continuously on $a,b,c$, we now see that \eqref{abc} holds whenever $a$ is sufficiently small, and similarly for $b$ and $c$.  Thus we may assume $a,b,c \geq \eps$ for some small absolute constant $\eps>0$.

We next consider the boundary case $a=1$, $b,c > \eps$.  Here, we claim strict inequality:
\begin{equation}\label{bcp}
 (b(1-c))^{1/p} + (c(1-b))^{1/p} < 1.
\end{equation}
Indeed, from the Cauchy-Schwarz inequality one has
$$
(b(1-c))^{1/2} + (c(1-b))^{1/2} \leq \left((b^{1/2})^2 + ((1-b)^{1/2})^2\right)^{1/2}\left(((1-c)^{1/2})^2 + (c^{1/2})^2\right)^{1/2} = 1
$$
and the claim follows since $\frac{1}{p} > \frac{1}{2}$.

For similar reasons, we obtain strict inequality in \eqref{abc} when $b=1$ or $c=1$.  By continuity, this establishes \eqref{abc} in all regions except the region
$$ \eps \leq a,b,c \leq 1-\eps$$
for some small absolute constant $\eps>0$.  We now work in this region.

We can rewrite \eqref{abc} as
$$ \left(\frac{1-c}{c}\right)^{1/p} + \left(\frac{1-a}{a}\right)^{1/p} + \left(\frac{1-b}{b}\right)^{1/p}  \leq \frac{1}{(abc)^{1/p}};$$
writing $x \coloneqq (\frac{1-a}{a})^{1/p}$, $y \coloneqq (\frac{1-b}{b})^{1/p}$, $z \coloneqq (\frac{1-c}{c})^{1/p}$, $x,y,z$ lies in the region
\begin{equation}\label{xyz}
 \left(\frac{\eps}{1-\eps}\right)^p \leq x,y,z \leq \left(\frac{1-\eps}{\eps}\right)^p
\end{equation}
and the above inequality transforms to
$$ x+y+z \leq \exp\left( \frac{f(x)+f(y)+f(z)}{p} \right)$$
or equivalently
\begin{equation}\label{fpx}
 f(x)+f(y)+f(z) - p \log(x+y+z) \geq 0
\end{equation}
where $f: (0,+\infty) \to (0,+\infty)$ is the function $f(z) \coloneqq \log(1+z^p)$.

Since the region \eqref{xyz} is compact, and the inequality \eqref{fpx} is already known on the boundary of this region, it suffices to verify \eqref{fpx} when $(x,y,z)$ is a critical point of the left-hand side, that is to say that
$$ f'(x) = f'(y) = f'(z) = \frac{p}{x+y+z}.$$
Since $f'(z) = \frac{p}{z+z^{1-p}}$, we can rewrite this condition as
\begin{equation}\label{xyo}
 x+x^{1-p} = y+y^{1-p} = z+z^{1-p} = x+y+z.
\end{equation}
The function $x \mapsto x+x^{1-p}$ is increasing for $x < (p-1)^{1/p}$ and decreasing for $x > (p-1)^{1/p}$, so it can only attain any given value at most twice.  From \eqref{xyo} and the pigeonhole principle, we conclude that at least two of $x,y,z$ are equal.  Without loss of generality we may assume $x=y$, then from \eqref{xyo} we have
$$ x^{1-p} = x+z$$
and
$$ z^{1-p} = 2x$$
and hence
$$ x^{1-p} = x + (2x)^{\frac{1}{1-p}}.$$
Dividing by $x$ we obtain
$$
x^{-p} = 1 + 2^{\frac{1}{1-p}} x^{\frac{p}{1-p}}$$
and then setting $u \coloneqq x^{-p}$ we conclude that
\begin{equation}\label{up}
 u = 1 + (u/2)^{\frac{1}{p-1}}.
\end{equation}
The function $u \mapsto 1 + (u/2)^{\frac{1}{p-1}}$ is convex, equals $2$ when $u=2$, $0$ when $u=0$, and is larger than $u$ for sufficiently large $u$.  As a consequence, the equation \eqref{up} has exactly two solutions, one at $u=2$ and one with $u > 2$; see Figure \ref{power-fig}. The second solution can be computed numerically as $u = 10.70297\dots$.  Thus, there are two critical points $(x,y,z)$ with $x=y$, the first of which is
$$(2^{-1/p}, 2^{-1/p}, 2^{-1/p}) = (0.67113\dots, 0.67113\dots, 0.67113\dots),$$
 and the second of which can be computed numerically as
$$ (x,y,z) = (0.25568\dots, 0.25568\dots, 2.48086\dots).$$
At the first critical point, we have $f(x)=f(y)=f(z) = \log(3/2)$, and one easily verifies that the left-hand side of \eqref{fpx} vanishes since $p = \log(27/4)/\log(3)$.  At the second critical point, one can numerically verify that
$$ f(x)=f(y) = 0.089321\dots; \quad f(z) = 1.766695\dots$$
and hence
$$ f(x)+f(y)+f(z)-p \log(x+y+z) = 0.040307\dots > 0$$
at this critical point, giving the claim.

\begin{figure} [t]
\centering
\includegraphics[width=5in]{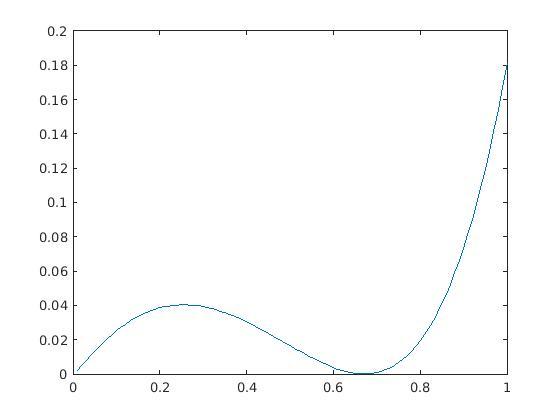}
\caption{A graph of $f(x)+f(y)+f(z)-p\log(x+y+z)$ for $0 \leq x \leq 1$ with $y=x$ and $z = x^{1-p} - x$.}
\label{wave-fig}
\end{figure}

\begin{figure} [t]
\centering
\includegraphics[width=5in]{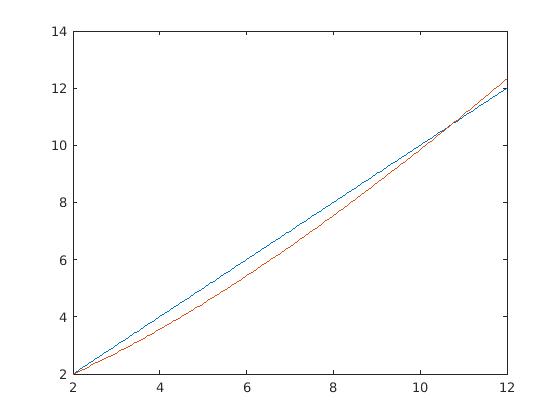}
\caption{A graph of $u$ (blue) and $1 + (u/2)^{\frac{1}{p-1}}$ (red) for $2 \leq u \leq 12$.}
\label{power-fig}
\end{figure}

\begin{remark}  The above methods extend\footnote{We thank Paata Ivanishvili for this comment.} to establish the more general bound
$$ f_1*\dots*f_k(1^n) \leq \|f_1\|_{\ell^p(\{0,1\}^n)} \dots \|f_k\|_{\ell^p(\{0,1\}^n)}$$
for any $k \geq 3$, where now $p \coloneqq \log_k \frac{k^k}{(k-1)^{k-1}}$.  In particular one has
$$ | \{ (A_1,\dots,A_{k-1},A) \in X \times X \times X: A = A_1 \uplus \dots \uplus A_{k-1} \} | \leq |X|^{k/p} $$
for any finite collection of sets.  We sketch the details as follows.  By repeating the above arguments (and using an induction on $k$ to handle boundary cases), one needs to show that
$$ f(x_1)+\dots+f(x_k) - p \log(x_1+\dots+x_k) \geq 0$$
for $\eps \leq x_1,\dots,x_k \leq 1-\eps$.  We can again restrict attention to critical points, in which
$$ x_1+x_1^{1-p} = \dots = x_k + x_k^{1-p} = x_1+\dots+x_k.$$
As before, $x_1,\dots,x_k$ can take only two values, say $x$ and $z$, leading to the equations
$$ x + x^{1-p} = z + z^{1-p} = ax + bz$$
for some positive integers $a,b$ summing to $k$.  Writing $u \coloneqq x^{-p}$ and $v \coloneqq z/x$, we have the system
\begin{align*}
u &= a-1 + bv \\ 
uv^{1-p} &= a + (b-1) v.
\end{align*}
Differentiating the second equation once with respect to $u$ gives
$$\frac{dv}{du} = \frac{v}{(b-1)v^p + (p-1) u} > 0$$
and differentiating twice gives (after some algebra)
$$ \frac{d^2 v}{du^2} = \frac{p-1}{v^p} \frac{dv}{du} \frac{(2-p)u}{(b-1)v^p + (p-1)u} > 0$$
so $v$ is again a convex function of $u$ (since $1 < p < 2$), and so as before the equation $u=a-1+bv$ has at most two solutions, including the one at $u=k-1$ and $v=1$.  Using the equation $x + x^{1-p} = ax+bz$ to implicitly define $z$ in terms of $x$, the function
$$ a f(x) + b f(z) - p \log(ax + bz)$$
then has two critical points, including one at $x=z=(k-1)^{-1/p}$ where the function vanishes.  Direct calculation shows that this critical point is a local minimum (basically because $f''((k-1)^{-1/p}) > 0$, which in turn follows from the inequality $p > \frac{k}{k-1}$), so the function must be positive at the other critical point (otherwise there would be an additional critical point from the mean value theorem and intermediate value theorem), giving the claim.
\end{remark}

\section{A variant for additive energy}

Recall from \cite{taovu} that the \emph{additive energy} $E(A)$ of a finite subset $A$ of an additive group $G$ is defined as the number of quadruples $(a_1,a_2,a_3,a_4) \in A \times A \times A \times A$ such that $a_1+a_2=a_3+a_4$.  We have the trivial bound $E(A) \leq |A|^3$, which is attained for instance when $A$ is itself a finite group.  By modifying the above arguments, we have the following refinement in the discrete cube $\{0,1\}^n$:

\begin{theorem}\label{energy}  Let $n \geq 0$, and let $A \subset \{0,1\}^n$.  Then $E(A) \leq |A|^p$, where $p \coloneqq \log_2 6 = 2.58496\dots$.  Furthermore, the exponent $p$ cannot be replaced by any smaller quantity.
\end{theorem}

The second claim is clear, since if $A = \{0,1\}^n$ then one easily computes that $|A|=|\{0,1\}|^n = 2^n$ and $E(A) = E(\{0,1\})^n = 6^n$.  As in the previous section, the theorem is proven by induction on $n$ together with an elementary inequality, namely

\begin{lemma}[Elementary inequality]\label{elem}  If $a_0,a_1 \geq 0$, then
$$ a_0^p + 4 (a_0 a_1)^{\frac{p}{2}} + a_1^{p} \leq (a_0+a_1)^{p}.$$
\end{lemma}

\begin{figure} [t]
\centering
\includegraphics[width=5in]{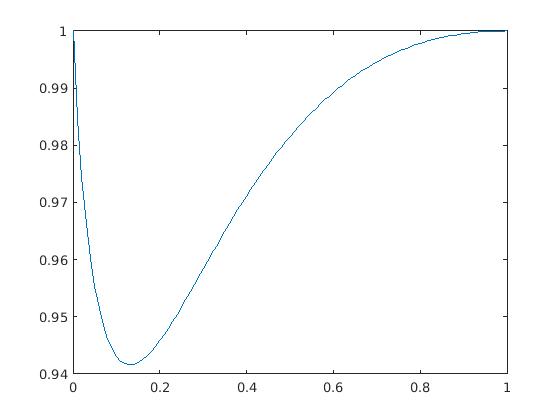}
\caption{A graph of $(x^p+4x^{p/2}+1)/(1+x)^p$ for $0 \leq x \leq 1$.}
\label{energy-fig}
\end{figure}

\begin{proof}  By symmetry and scaling we may assume that $a_1=1$ and $a_0 = x \in [0,1]$, thus we need to show that
$$ x^p + 4 x^{p/2} + 1 \leq (1+x)^p$$
for $0 \leq x \leq 1$ (see Figure \ref{energy-fig}).  Near $x=0$, the left-hand side is $1 + O( x^{p/2} )$ and the right-hand side is at least $1+px$, so the claim holds for $x$ sufficiently close to zero.  At $x=1$, the function $x^p + 4x^{p/2}+1$ takes the value of $6$, first derivative of $3p$, and second derivative of
$$ p(p-1) + p (p-2) = 5.60917\dots,$$
while $(1+x)^p$ takes the value of $2^p = 6$, first derivative of $p 2^{p-1} = 3p$, and second derivative of
$$ p(p-1) 2^{p-2} = 6.14560\dots,$$
so the claim also holds for $x$ sufficiently close to $1$.  It thus suffices to verify the inequality at any critical point of the functional
$$ \frac{x^p + 4x^{p/2} + 1}{(1+x)^p}$$
in $0 < x < 1$.  Differentiating, we see that such a critical point solves the equation
$$ (p x^{p-1} + 2 p x^{(p-2)/2}) (1+x)^p = (x^p + 4x^{p/2} + 1) p (1+x)^{p-1}$$
which simplifies to
\begin{equation}\label{x1}
 x^{p-1} + 2 x^{(p-2)/2} - 2 x^{p/2} = 1.
\end{equation}
The second derivative of the left-hand side is
$$ \frac{p-2}{2} x^{\frac{p-6}{2}} ( (2p-2) x^{p/2} + p-4 - p x );$$
since $(2p-2)x^{p/2} \leq (2p-2)x \leq px$ and $p<4$, we conclude that $x^{p-1} + 2 x^{(p-2)/2} - 2 x^{p/2}$ is strictly concave.  As this function is $0$ at $x=0$ and $1$ at $x=1$, and has a derivative of $p-3 < 0$ at $x=1$, there are exactly two solutions to \eqref{x1} for $0 \leq x \leq 1$, one at $x=1$ and another with $0 < x < 1$; see Figure \ref{energy2-fig}.  The second solution can be numerically evaluated as $x = 0.131657\dots$, at which
$$ x^p + 4x^{p/2} + 1 = 1.29634\dots$$
and
$$ (1+x)^p = 1.376738\dots$$
giving the claim.
\end{proof}

\begin{figure} [t]
\centering
\includegraphics[width=5in]{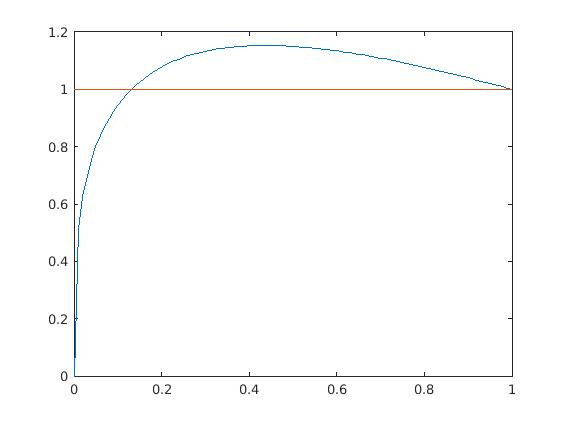}
\caption{A graph of $x^{p-1} + 2 x^{(p-2)/2} - 2 x^{p/2}$ for $0 \leq x \leq 1$.}
\label{energy2-fig}
\end{figure}

Now we establish Theorem \ref{energy}.  The claim is trivial for $n=0$, so suppose that $n \geq 1$ and that the claim has already been proven for $n-1$.  For $A \subset \{0,1\}^n$, we may partition
$$ A = (A_0 \times \{0\}) \uplus (A_1 \times \{1\})$$
for some $A_0,A_1 \subset \{0,1\}^n$.  We can then split
$$ E(A) = E(A_0) + 4 | \{ (a_0,a_1,a'_0,a'_1) \in A_0 \times A_1 \times A_0 \times A_1: a_0 + a_1 = a'_0 + a'_1 \}| + E(A_1).$$
By the Cauchy-Schwarz inequality (and writing $a_0+a_1 = a'_0 + a'_1$ as $a_0-a'_0 = a'_1 - a_1$) we have
$$ | \{ (a_0,a_1,a'_0,a'_1) \in A_0 \times A_1 \times A_0 \times A_1: a_0 + a_1 = a'_0 + a'_1 \}| \leq E(A_0)^{1/2} E(A_1)^{1/2}$$
and hence by the induction hypothesis
$$ E(A) \leq |A_0|^p + 4 (|A_0| |A_1|)^{p/2} + |A_1|^p.$$
Applying Lemma \ref{elem} and noting that $|A_0|+|A_1| = |A|$, we obtain $E(A) \leq |A|^p$, closing the induction.

\begin{remark}  The same argument shows that
$$ \| f*f \|_{\ell^2} \leq \|f\|_{\ell^{4/p}}^2$$
for any function $f: \{0,1\}^n \to \C$ (where the convolution $f*f$ is viewed as a function on $\{0,1,2\}^n$).  By several applications of the Cauchy-Schwarz inequality, this implies that
$$ |\int_{a_1,a_2,a_3,a_4 \in \{0,1\}^n: a_1 + a_2 = a_3 + a_4} f_1(a_1) f_2(a_2) f_3(a_3) f_4(a_4)|\leq \|f_1\|_{\ell^{4/p}} \|f_2\|_{\ell^{4/p}} \|f_3\|_{\ell^{4/p}} \|f_4\|_{\ell^{4/p}} $$
for any functions $f_1,f_2,f_3,f_4: \{0,1\}^n \to \C$. Thus, for instance, if $A_1,A_2,A_3,A_4 \in \{0,1\}^n$, the number of solutions to $a_1+a_2=a_3+a_4$ with $a_1 \in A_1, a_2 \in A_2, a_3 \in A_3, a_4 \in A_4$ is at most $|A_1|^{p/4} |A_2|^{p/4} |A_3|^{p/4} |A_4|^{p/4}$.
\end{remark}

\begin{remark} In \cite{bl}, the method of compressions is used to obtain optimal lower bounds for the size $|A+B|$ of a sumset of two subsets $A,B$ of $\{0,1\}^n$ of specified cardinality.  It is possible that compression methods could also be used to obtain an alternate proof of Theorem \ref{energy}, and perhaps to also refine the upper bound of $|A|^{\log_2 6}$ slightly when $|A|$ is not a power of two.  However, we were unable to use the method of compressions to attack Theorem \ref{main}.
\end{remark}

\section{Acknowledgments}

The authors would like to thank Siavash Mirarab for bringing this problem to our attention, and David Speyer and the anonymous referee for helpful comments. DK is supported by NSF award CCF-1553288 (CAREER).  TT is supported by NSF grant DMS-1266164, the James and Carol Collins Chair, and by a Simons Investigator Award.


\end{document}